\documentclass[a4paper,11pt]{article}
\usepackage{enumitem}   
\usepackage{cite}
\usepackage{tikz-cd}
\usepackage{amsmath,amssymb, amsthm,graphicx,yfonts,mathrsfs,color}
\usepackage[utf8]{inputenc}
\usepackage{braket}
\usepackage{fullpage}

\makeatletter
\newcommand{\subalign}[1]{%
  \vcenter{%
    \Let@ \restore@math@cr \default@tag
    \baselineskip\fontdimen10 \scriptfont\tw@
    \advance\baselineskip\fontdimen12 \scriptfont\tw@
    \lineskip\thr@@\fontdimen8 \scriptfont\thr@@
    \lineskiplimit\lineskip
    \ialign{\hfil$\m@th\scriptstyle##$&$\m@th\scriptstyle{}##$\crcr
      #1\crcr
    }%
  }
}
\makeatother

\theoremstyle{plain}
\newtheorem{thm}{Theorem}[section]
\newtheorem{prop}[thm]{Proposition}
\newtheorem{lemma}[thm]{Lemma}

\newtheorem{cor}[thm]{Corollary}
\theoremstyle{definition}

\newtheorem{rmk}[thm]{Remark}

\makeatletter
\def\th@plain{%
  \thm@notefont{}% same as heading font
  \itshape % body font
}
\def\th@definition{%
  \thm@notefont{}% same as heading font
  \normalfont % body font
}
\makeatother
\author{Francesca Bianchi}
\title{Consequences of the functional equation of the $p$-adic $L$-function of an elliptic curve}
\date{\today}

\begin{document}
\maketitle

\begin{abstract}
We prove that the first two coefficients in the series expansion around $s=1$ of the $p$-adic $L$-function of an elliptic curve over $\mathbb{Q}$ are related by a formula involving the conductor of the curve. This is analogous to a recent result of Wuthrich for the classical $L$-function \cite{subleading}, which makes use of the functional equation. We present a few other consequences for the $p$-adic $L$-function and a generalisation to the base-change to an abelian number field.
\end{abstract}

\section{Introduction}
In \cite{subleading}, Wuthrich proves a relation between the leading and sub-leading coefficients of the Taylor series expansion at $s=1$ of the $L$-function of an elliptic curve over a number field, under the assumptions that the series admits an analytic continuation to $s=1$ and satisfies a functional equation. Wuthrich remarks that a similar result can potentially be obtained for any $L$-function which satisfies a functional equation.

In this paper, we consider the $p$-adic $L$-function of an elliptic curve over $\mathbb{Q}$ and its twist by a quadratic character and explore the analogue of Wuthrich's result. The prime $p$ will be assumed to be odd and of semistable reduction for the curve. We start in Section \ref{notation} by reviewing the definition of the $p$-adic $L$-function and of its twists by Dirichlet characters $\psi$. Indeed, we define two $p$-adic $L$-functions, namely $L_p(E,\alpha,\psi,s)$, which is a function of $s\in\mathbb{Z}_p$, and $\mathcal{L}_p(E,\alpha,\psi,T)$, a power series in $T$. We state in Proposition \ref{functionalsupergenerals} and Corollary \ref{functionalsupergeneralT} the functional equation for $L_p(E,\alpha,\psi,s)$ and the corresponding result for $\mathcal{L}_p(E,\alpha,\psi,T)$. Proofs of the interpolation property and of the functional equation satisfied, more generally, by the twists of the $p$-adic $L$-function of a weight $k$ cuspidal eigenform of Nebentypus $\epsilon$ are provided in \cite[Chapter 1, $\S\S 13,17$]{mtt}.

Subsequently, in Section \ref{relations}, we prove the analogue of \cite[Theorem 1]{subleading} for $L_p(E,\alpha,\psi,s)$ when $\psi$ is a real character. In particular, we prove in Theorem \ref{mains} that the sub-leading coefficient in the Taylor expansion of $L_p(E,\alpha,\psi,s)$ at $s=1$ is equal to the leading coefficient multiplied by a constant dependent exclusively on the conductor of $E$, on the conductor of $\psi$ and on the prime $p$. The assumption that $\psi$ is a real character is necessary, since the functional equation relates the values of $L_p(E,\alpha,\psi,s)$ to those of $L_p(E,\alpha,\bar{\psi},2-s)$.

Although the latter result is perhaps of some interest on its own, when studying $p$-adic $L$-functions one is often concerned about the corresponding power series $\mathcal{L}_p(E,\alpha,\psi,T)$, because of its central importance in Iwasawa theory of elliptic curves. In Theorem \ref{main} we thus derive a similar result for the first two leading coefficients of $\mathcal{L}_p(E,\alpha,\psi,T)$. In Section \ref{relationT}, we also prove another consequence of the relation between $\mathcal{L}_p(E,\alpha,\psi,T)$ and $\mathcal{L}_p(E,\alpha,\bar{\psi}, (1+T)^{-1}-1)$, namely that the $\mu$-invariant of $\mathcal{L}_p(E,\alpha,\psi,T)$ is the same as that of $\mathcal{L}_p(E,\alpha,\bar{\psi},T)$. This is Theorem \ref{mupsipsibar}.

We end the article with a generalisation of some of the results to the $p$-adic $L$-function of the base-change of an elliptic curve to an abelian number field.\\

{\bf Acknowledgements.} The author is very grateful to her adviser Jennifer Balakrishnan for suggesting the question and for proof-reading carefully. She would also like to thank Chris Wuthrich for reading an earlier version of this article and providing extremely useful comments. Finally, she thanks Thanasis Bouganis for the very helpful conversations.

\section{Notation and statement of the functional equation}
\label{notation}
In this section, we recall the definition of the $p$-adic $L$-function and the statement of the functional equation it satisfies. We follow the notation of \cite{steinwuth} for the modular symbols and measures as well as the construction of the $p$-adic $L$-function, although we also allow non-trivial twists. We fix once and for all embeddings of $\bar{\mathbb{Q}}$ into $\mathbb{C}$ and $\bar{\mathbb{Q}}_p$.

Let $E/\mathbb{Q}$ be an elliptic curve of conductor $N$ and assume that $p$ is an odd prime at which $E$ has semistable reduction. Let $f(z)$ be the newform of weight 2 corresponding to $E$ by modularity and denote its $q$-expansion $\sum_{n=1}^{\infty} a_n q^n$.

We now define a quantity $\alpha$, dependent on the reduction type of $E$ at $p$. In particular, if $E$ has good ordinary reduction, we let $\alpha$ be the unique $p$-adic unit in $\mathbb{Z}_p$ which is a root of the characteristic polynomial of Frobenius $P(X)=X^2-a_pX+p$. When $E$ has good supersingular reduction, we denote by $\alpha$ an arbitrarily fixed root of $P(X)$. Finally, if $E$ has multiplicative reduction, we set $\alpha=a_p$, so that, in particular, $\alpha=1$ (resp. $\alpha=-1$) if the reduction type is split (resp. non-split). We call $\alpha$ an \emph{allowable $p$-root (for $E$)} and we emphasize that $\alpha$ is uniquely determined, unless $E$ has good supersingular reduction at $p$.

For $r\in \mathbb{Q}$, we let
\begin{equation*}
\lambda^{\pm}(r)=\mp\pi i \left(\int_{r}^{i\infty}f(z)dz\pm\int_{-r}^{i\infty}f(z)dz\right).
\end{equation*}
Normalising $\lambda^{+}(r)$ (resp. $\lambda^{-}(r)$) by the real period $\Omega_E^+$ (resp. imaginary period $\Omega_E^-$) of $E$ gives the rational numbers $[r]^{\pm}=\frac{\lambda^{\pm}(r)}{\Omega^{\pm}_E}$.

Any unit $x$ in $\mathbb{Z}_p$ has a unique decomposition as a product of the form $\omega(x)\braket{x}$ where $\omega(x)$ is a $(p-1)^{\text{st}}$ root of unity and $\braket{x}\in 1+p\mathbb{Z}_p$ (note the running assumption that $p$ is odd).

As in \cite[I \S\S 11, 13]{mtt}, for a positive integer $M$, we let
$
\mathbb{Z}^*_{p,M}=\varprojlim_{n}(\mathbb{Z}/p^n M\mathbb{Z})^*
$
and define a \emph{$p$-adic character} as a continuous homomorphism $\mathbb{Z}_{p,M}^*\to \mathbb{C}_p^*$. In particular, since we have fixed an embedding of $\bar{\mathbb{Q}}$ into $\bar{\mathbb{Q}}_p$, we may regard the notion of a finite order $p$-adic character and a Dirichlet character as equivalent and hence, even in the $p$-adic setting, we may use the notation $\bar{\psi}$ for the inverse of the finite order character $\psi$ and say $\psi$ is \emph{real-valued} if it is real-valued as a Dirichlet character, i.e. $\psi=\bar{\psi}$.

We then define the measures
\begin{equation*}
\mu^{\pm}_{\alpha}(a+p^k M \mathbb{Z}_p)=\frac{1}{\alpha^k}\left[\frac{a}{M p^k}\right]^{\pm} -\frac{(1-\delta)}{\alpha^{k+1}}\left[\frac{a}{M p^{k-1}}\right]^{\pm},
\end{equation*}
where $\delta=0$ if $E$ has good reduction and $1$ otherwise.

Let $\psi$ be a finite order $p$-adic character (possibly the identity) of conductor $p^kM$, with $(M,p)=1$, and write $\braket{x}^{s-1}=\exp{((s-1)\log{\braket{x}})}=\sum_{n=0}^{\infty}\frac{(s-1)^n}{n!}(\log{\braket{x}})^n$, where $\exp$ and $\log$ are the $p$-adic exponential and logarithm maps. Then we define 
\begin{equation*}
L_p(E,\alpha,\psi,s)=\int_{\mathbb{Z}_{p,M}^*}\psi(x)\braket{x}^{s-1}d \mu_{\alpha}(x),
\end{equation*}
where $\mu_{\alpha}=\mu_{\alpha}^+$ if the sign of $\psi$ is $+1$ and $\mu_{\alpha}=\mu_{\alpha}^-$ otherwise. The $p$-adic $L$-function $L_p(E,\alpha,\psi,s)$ is a locally analytic function of $s\in\mathbb{Z}_p$ (see the first proposition in \cite[I \S 13]{mtt}).

We now perform a change of variable in the definition of $L_p(E,\alpha,\psi,s)$. Consider the extension $\mathbb{Q}(\mu_{p^{\infty}})$ of $\mathbb{Q}$, obtained by adjoining to $\mathbb{Q}$ all roots of unity of $p$-power order. Furthermore, let $\Gamma$ denote the Galois group of the cyclotomic $\mathbb{Z}_p$-extension $\mathbb{Q}^{\infty}$ of $\mathbb{Q}$ and let $\mathbb{Q}^n$ be the $n^{\text{th}}$ layer in the cyclotomic tower.\\
Then there is a canonical homomorphism
\begin{equation*}
\kappa: \text{Gal}(\mathbb{Q}(\mu_{p^{\infty}})/\mathbb{Q})\to \mathbb{Z}_p^*
\end{equation*}
defined by the relation $\sigma(\zeta)=\zeta^{\kappa(\sigma)}$ for all $\sigma\in \text{Gal}(\mathbb{Q}(\mu_{p^{\infty}})/\mathbb{Q})$ and $\zeta\in\mu_{p^{\infty}}$. The homomorphism $\kappa$ is called the \emph{cyclotomic character} and it induces an isomorphism $\Gamma\to 1+p\mathbb{Z}_p$. We may choose a topological generator $\gamma\in \Gamma$ and define $T=\kappa(\gamma)^{s-1}-1$. The result is a power series in $T$, with coefficients in $\mathbb{Q}_p[\psi](\alpha)$,
\begin{equation*}
\mathcal{L}_p(E,\alpha,\psi, T)=\int_{\mathbb{Z}_{p,M}^*}\psi(x)(1+T)^{\frac{\log{\braket{x}}}{\log{\kappa(\gamma)}}} d \mu_{\alpha}(x).
\end{equation*}
When the character $\psi$ is trivial, we write simply $L_p(E,\alpha, s)$ and $\mathcal{L}_p(E,\alpha, T)$.

Denote by $Q$ the largest divisor of $N$ coprime to $pM$. Then there exists $c_Q=\pm 1$ such that $w_Q(f)= c_Q f$, where 
\begin{equation*}
w_Q(f)=f\vert_{\left( \begin{smallmatrix} Qa & b \\
 Nc & Qd \end{smallmatrix} \right)}(z)
\end{equation*}
for some integers $a,b,c,d$ satisfying $Qad-(N/Q)bc=1$. Note that, when $Q=N$, we have $w_E=-c_Q$ where $w_E$ is the sign in the functional equation of the complex $L$-function $L(E,s)$.
\begin{prop}[Functional equation]
\label{functionalsupergenerals}
Let $E$ be an elliptic curve of conductor $N$ with semistable reduction at the odd prime $p$ and let $\psi$ be a Dirichlet character of conductor $p^kM$, for some non-negative integers $k$ and $M$, with $(M,p)=1$. For each allowable $p$-root $\alpha$, the $p$-adic $L$-function $L_p(E,\alpha,\psi,s)$ satisfies the functional equation
\begin{equation*}
L_p(E,\alpha,\psi,2-s)=-c_Q\cdot\bar{\psi}(-Q)\braket{Q}^{s-1}L_p(E,\alpha,\bar{\psi},s),
\end{equation*}
where $Q$ is the largest divisor of $N$ coprime to $pM$.
\end{prop}
\begin{proof}
See \cite[I \S17, Corollary 2]{mtt}.
\end{proof}

\begin{cor}
\label{functionalsupergeneralT}
With the assumptions as in Proposition \ref{functionalsupergenerals}, we have
\begin{equation*}
\mathcal{L}_p(E,\alpha,\psi,T)=-c_Q\cdot\bar{\psi}(-Q)\braket{Q}^{\frac{\log(1+T)^{-1}}{\log\kappa(\gamma)}}\mathcal{L}_p(E,\alpha,\bar{\psi},(1+T)^{-1}-1).
\end{equation*}
\end{cor}

\section{Relation between the leading and sub-leading coefficients of the $p$-adic $L$-function of an elliptic curve}
\label{relations}
We adopt the same notation as in Section \ref{notation}.
\begin{lemma}
\label{ordchar}
Let $\psi$ be a real-valued Dirichlet character. Then
\begin{equation*}
-c_Q\cdot \psi(-Q)=(-1)^{\text{ord}_{s=1}L_p(E,\alpha,\psi,s)}.
\end{equation*}
\end{lemma}
\begin{proof}
This is a standard exercise given the functional equation. Indeed, since $\psi$ is assumed to be real-valued, Proposition \ref{functionalsupergenerals} becomes
\begin{equation}
\label{lambdapfunctional}
\Lambda_p(E,\alpha,\psi,2-s)=-c_Q\cdot\psi(-Q)\Lambda_p(E,\alpha,\psi,s),
\end{equation}
where we put $\Lambda_p(E,\alpha,\psi,s)=f(s)\cdot L_p(E,\alpha,\psi,s)$ with $f(s)=\braket{Q}^{s/2}$. Let $m=\text{ord}_{s=1}L_p(E,\alpha,\psi,s)$, so that $\frac{d^i}{ds^i}L_p(E,\alpha,\psi,s)\big\vert_{s=1}=0$ for all $i<m$ and $\frac{d^m}{ds^m}L_p(E,\alpha,\psi,s)\big\vert_{s=1}\neq 0$. Then differentiating both sides of (\ref{lambdapfunctional}) and evaluating at $s=1$ yields
\begin{equation*}
((-1)^{m}+c_Q\cdot \psi(-Q))\frac{d^m}{ds^m}\Lambda_p(E,\alpha,\psi,s)\big\vert_{s=1}=0.
\end{equation*}
By the definition of $m$, one has 
\begin{equation}
\frac{d^m}{ds^m}\Lambda_p(E,\alpha,\psi,s)\big\vert_{s=1}=f(1)\frac{d^m}{ds^m}L_p(E,\alpha,\psi,s)\big\vert_{s=1}\neq 0,
\end{equation}
which forces $-c_Q\cdot \psi(-Q)=(-1)^m$.
\end{proof}

\begin{rmk}
Assume that the conductor of $\psi$ is not divisible by any prime of additive reduction for $E$. A `refined' version of the $p$-adic Birch and Swinnerton-Dyer conjecture then predicts that
\begin{equation*}
\text{ord}_{s=1}L(E,\alpha,\psi,s)=m(E,\psi)+\delta,
\end{equation*}
where $\psi$ is now viewed as a one-dimensional representation of $\text{Gal}(K/\mathbb{Q})$ for some number field $K$, $m(E,\psi)$ is the multiplicity of $\psi$ in the decomposition of $E(K)\otimes_{\mathbb{Z}}\mathbb{C}$ into irreducible representations of $\text{Gal}(K/\mathbb{Q})$ and $\delta$ vanishes except for when the $p$-adic multiplier of the pair $(\alpha,\psi)$ is zero, in which case $\delta=1$.\\
In particular, when $\psi$ is trivial, $m(E,\psi)=\text{rank} E(\mathbb{Q})$.
\end{rmk}
We now state and prove the result analogous to Wuthrich \cite[Theorem 1]{subleading}.

\begin{thm}
\label{mains}
Let $E$ be an elliptic curve over $\mathbb{Q}$ and $p$ an odd prime of semistable reduction for $E$. Let $\psi$ be either the trivial character or a quadratic character of conductor $p^k M$. Let $\alpha$ be an allowable $p$-root and denote by $m$ the order of vanishing at $s=1$ of $L_p(E,\alpha,\psi,s)$. Then
\begin{equation*}
L_p(E,\alpha,\psi,s)=a_m(s-1)^m+a_{m+1}(s-1)^{m+1}+\cdots
\end{equation*}
with
\begin{equation*}
a_{m+1}=-\frac{1}{2}\log\braket{Q} a_m.
\end{equation*}
\end{thm}

\begin{proof}
By Lemma \ref{ordchar}, we have
\begin{equation*}
\Lambda_p(E,\alpha,\psi,2-s)=(-1)^m\Lambda_p(E,\alpha,\psi,s),
\end{equation*}
where, as in the proof of the lemma, $\Lambda_p(E,\alpha,\psi,s)=f(s)\cdot L_p(E,\alpha,\psi,s)$ with $f(s)=\braket{Q}^{s/2}$.
We can now follow Wuthrich's proof nearly identically.\\
In particular, for $i\equiv m+1 \text{ mod } 2$, one has
\begin{equation}
\label{vanishderivative}
\frac{d^i}{ds^i}\Lambda_p(E,\alpha,\psi,s)\big\vert_{s=1}=0.
\end{equation}
Choosing $i=m+1$, this gives
\begin{align*}
0&=\frac{d^{m+1}}{ds^{m+1}}\left(f(s)L_p(E,\alpha,\psi,s)\right)\big\vert_{s=1}\\
&=\left(\sum_{k=0}^{m+1}\binom{m+1}{k}\frac{d^{m+1-k}}{ds^{m+1-k}}f(s) \cdot \frac{d^{k}}{ds^{k}}L_p(E,\alpha,\psi,s)\right)\bigg\vert_{s=1}\\
&=\left((m+1)f'(s)\cdot \frac{d^{m}}{ds^{m}} L_p(E,\alpha,\psi,s)+f(s)\cdot \frac{d^{m+1}}{ds^{m+1}}L_p(E,\alpha,\psi,s)\right)\bigg\vert _{s=1}\\
&=(m+1)!(f'(1) \cdot a_m+f(1) \cdot a_{m+1}).
\end{align*}
Therefore, $a_{m+1}=-\frac{f'(1)}{f(1)}a_m=-\frac{1}{2}\log\braket{Q} a_m$.
\end{proof}
Let $E$ be as in Theorem \ref{mains}. The twist of $E$ by a quadratic character $\psi$ of conductor coprime to $p$ can be realised as an elliptic curve $E'/\mathbb{Q}$ which has semistable reduction at $p$. Therefore, if the conductor of $\psi$ is coprime to $p$, the statement of Theorem \ref{mains} for $L_p(E,\alpha,\psi,s)$ is implicitly included in the statement of Theorem \ref{mains} for the $p$-adic $L$-function of $E'$ with trivial twist.\\
On the other hand, the twist of $E$ by a quadratic character of conductor divisible by $p$ is an elliptic curve $E'$ with additive reduction at $p$. In other words, the quadratic character case in Theorem \ref{mains} is not completely exhausted by the result for the trivial character.

Note that Wuthrich's method can be generalised also to obtain the coefficient $a_{m+k}$ for every positive odd $k$ in terms of the coefficients $a_{m+j}$, where $0\leq j< m$, just by considering (\ref{vanishderivative}) for $i=m+k$. Recursively, this gives a way of expressing $a_{m+k}$ for \emph{odd} $k$ in terms of $a_{m+\ell}$ for \emph{even} $\ell$ with $\ell<k$.
In particular, we may consider Theorem \ref{mains} as a special case of the following result.

\begin{thm}
\label{mainsgeneral}
Let $E$ be an elliptic curve over $\mathbb{Q}$ and $p$ an odd prime of semistable reduction for $E$. Let $\psi$ be either the trivial character or a quadratic character of conductor $p^k M$. Denote by $m$ the order of vanishing at $s=1$ of $L_p(E,\alpha,\psi,s)$. Then
\begin{equation*}
L_p(E,\alpha,\psi,s)=a_m(s-1)^m+a_{m+1}(s-1)^{m+1}+\cdots,
\end{equation*}
where, for odd $k>0$, we have
\begin{equation*}
a_{m+k}=-\sum_{i=0}^{k-1}\frac{1}{(k-i)!}\cdot \left(\frac{\log \braket{Q}}{2}\right)^{k-i}a_{m+i}.
\end{equation*}
Therefore, for odd $k$, $a_{m+k}$ can be expressed as a linear combination of $a_{m+\ell}$ for even $\ell$, $0\leq\ell< k$.
\end{thm}

\begin{proof}
Using (\ref{vanishderivative}) with $i=m+k$ and the fact that $\frac{d^{k}}{ds^k}L_p(E,\alpha,\psi,s)\big\vert_{s=1}=0$ for all $k<m$ gives
\begin{equation*}
\sum_{r=m}^{m+k}\frac{(m+k)!}{(m+k-r)!} f^{(m+k-r)}(1)a_{r}=0.
\end{equation*}
It remains to notice that, for all $j\geq 0$, $f^{(j)}(s)=\left(\frac{\log\braket{Q}}{2}\right)^{j}\braket{Q}^{s/2}$.
\end{proof}

\begin{rmk}
There are essentially two reasons why we did not simply state Theorem \ref{mains} as a corollary of Theorem \ref{mainsgeneral}. The first one is that we wanted to mimic the proof in \cite{subleading} as much as possible. The second one is that, although Theorem \ref{mainsgeneral} is more general than Theorem \ref{mains}, the latter is perhaps the more interesting result. Indeed, the $p$-adic version of the Birch and Swinnerton-Dyer conjecture (see Conjecture (BSD($p$)) I. in \cite[II \S 10]{mtt}) gives an explicit expression for $a_m$, involving, among other factors, the order of the Tate-Shafarevich group.\\
Thus, a similar observation to Wuthrich's in \cite{subleading} also holds here, i.e. what Theorem \ref{mains} really states is that, if BSD($p$) holds, then not only the leading but also the sub-leading coefficient of $L_p(E,\alpha,s)$ is known.\\
Under some extra assumptions, a similar result follows for the quadratic character case (see Conjecture (BSD($p$,$\psi$)) I. in \cite[II \S 11]{mtt}). 
\end{rmk}

\section{Results for $\mathcal{L}_p(E,\alpha,\psi,T)$}
\label{relationT}
The starting point for this section is Corollary \ref{functionalsupergeneralT}.  The first result is then
\begin{thm}
\label{main}
Let $E$ be an elliptic curve over $\mathbb{Q}$ and $p$ an odd prime of semistable reduction for $E$. Let $\psi$ be either the trivial character or a quadratic character of conductor $p^k M$. Let $\alpha$ be an allowable $p$-root and denote by $m$ the order of vanishing at $T=0$ of ${\mathcal{L}}_p(E,\alpha,\psi,T)$. Then
\begin{equation*}
\mathcal{L}_p(E,\alpha,\psi,T) = c_m T^m+ c_{m+1} T^{m+1}+\cdots
\end{equation*}
where 
\begin{equation*}
c_{m+1} = -\frac{c_m}{2}\left(\frac{\log\braket{Q}}{\log(\kappa(\gamma))}+m\right).
\end{equation*}
\end{thm}
\begin{proof}
By Corollary \ref{functionalsupergeneralT}, we have
\begin{equation}
\label{functionalgenT}
\mathcal{L}_p(E,\alpha,\psi,(1+T)^{-1}-1)=(-c_Q)\psi(-Q)(1+T)^{\frac{\log\braket{Q}}{\log \kappa(\gamma)}}\mathcal{L}_p(E,\alpha,\psi,T).
\end{equation}
Differentiating $m$ times and evaluating at $T=0$ yields the equality $-c_Q\psi(-Q)=(-1)^m$, similarly to Lemma \ref{ordchar}.\\
Then, differentiating $m+1$ times the right hand side resp. the left hand side of (\ref{functionalgenT}) and evaluating at $T=0$ gives
\begin{equation*}
(-1)^m(m+1)!\left(\frac{\log\braket{Q}}{\log \kappa(\gamma)}c_m+c_{m+1}\right)
\end{equation*}
resp.
\begin{equation*}
(-1)^{m+1}(m+1)!(c_{m+1}+mc_m).
\end{equation*}
\end{proof}
Assume now that $E$ has ordinary semistable reduction at the odd prime $p$, that is, either good ordinary or multiplicative reduction. Since, in this case, there is a unique $p$-admissible root $\alpha$ for $E$, we may drop $\alpha$ from the notation and simply write $\mathcal{L}_p(E,\psi,T)$ for the $p$-adic $L$-function of $E$, twisted by $\psi$. By \cite{wuthrichintegral}, $\mathcal{L}_p(E,\psi,T)\in\Lambda= \mathbb{Z}_p[\psi][[T]]$ and we define the \emph{$\mu$-invariant} of $\mathcal{L}_p(E,\alpha,\psi,T)$ as the $p$-adic valuation of the highest power of $\pi$ dividing $\mathcal{L}_p(E,\alpha,\psi,T)$, where $\pi$ is a uniformiser for $\mathbb{Z}_p[\psi]$.\\
The relation between $\mathcal{L}_p(E,\alpha,\psi,T)$ and $\mathcal{L}_p(E,\alpha,\bar{\psi},T)$ allows us to derive a result concerning the relation between the $\mu$-invariant of $\mathcal{L}_p(E,\alpha,\psi,T)$ and that of $\mathcal{L}_p(E,\alpha,\bar{\psi},T)$.

\begin{thm}
\label{mupsipsibar}
Let $E$ be an elliptic curve over $\mathbb{Q}$ and $p$ an odd prime of ordinary semistable reduction for $E$. Let $\psi$ be a character of conductor $p^k M$. The $\mu$-invariant of $\mathcal{L}_p(E,\psi,T)$ equals the $\mu$-invariant of $\mathcal{L}_p(E,\bar{\psi},T)$.
\end{thm} 
\begin{proof}
By Corollary \ref{functionalsupergeneralT}, 
\begin{equation*}
\mathcal{L}_p(E,\psi,T)=-c_Q\cdot\bar{\psi}(-Q)\braket{Q}^{\frac{\log(1+T)^{-1}}{\log\kappa(\gamma)}}\mathcal{L}_p(E,\bar{\psi},(1+T)^{-1}-1).
\end{equation*}
Now $c_Q=\pm 1$, $\bar{\psi}(-Q)$ is a root of unity and $\braket{Q}^{\frac{\log(1+T)^{-1}}{\log\kappa(\gamma)}}$ is a unit power series. Therefore the $\mu$-invariant of $\mathcal{L}_p(E,\psi,T)$ is the same as the $\mu$-invariant of $\mathcal{L}_p(E,\bar{\psi},(1+T)^{-1}-1)$.\\
It remains to show that the $\mu$-invariant of a power series in $\Lambda$ is unchanged under the substitution $T\mapsto(1+T)^{-1}-1$. For this purpose, let $f(T)=\sum_{k=0}^{\infty}a_kT^k$ and $g(T)= f((1+T)^{-1}-1)$. Then $g(T)=\sum_{k=0}^{\infty} b_kT^k$, where
\begin{equation*}
b_0=a_0 \qquad \text{and}\qquad b_k=(-1)^k\sum_{i=0}^{k-1}\binom{k-1}{i}a_{i+1} \text{ for } k\geq 1.
\end{equation*}
This comes from the fact that, for $j\geq 1$, we have the equality of formal power series $((1+T)^{-1}-1)^j=\sum_{k=j}^{\infty}\binom{k-1}{j-1}(-1)^kT^k$.\\
Let now $k$ be the minimum integer such that the $\mu$-invariant of $f(T)$ equals the $p$-adic valuation of $a_k$, denoted $\nu_p(a_k)$. Then
\begin{equation*}
\nu_p(b_0)=\nu_p(a_0)\geq\nu_p(a_k), \qquad \nu_p(b_i)\geq \min_{1\leq j\leq i}\nu_p (a_j)\geq \nu_p(a_k) \quad \forall\ i\geq 1
\end{equation*}
and
\begin{equation*}
\nu_p(b_{k})=\nu_p(a_k).
\end{equation*}
This completes the proof.
\end{proof}

\section{Generalisations to the base-change of an elliptic curve}
Let $E$ be an elliptic curve over $\mathbb{Q}$ of conductor $N$. Let $K$ be an abelian number field such that the additive primes for $E/\mathbb{Q}$ remain additive in the base-change of $E$ to $K$. Denote by $p$ an odd prime of semistable reduction for $E$ and assume further that the field $K$ is disjoint from the cyclotomic $\mathbb{Z}_p$-extension of $\mathbb{Q}$. Fix a choice of an allowable $p$-root $\alpha$ for $E/\mathbb{Q}$ as in Section \ref{notation}. We can then define the $p$-adic $L$-function for the base-changed curve $E/K$ as (cf. \cite{mazurswinn} and \cite{matsuno})
\begin{equation}
\label{productformula}
\mathcal{L}_p(E/K,\alpha,T)=\prod_{\psi\in \widehat{\text{Gal}(K/\mathbb{Q})}}\mathcal{L}_p(E,\alpha,\psi,T),
\end{equation}
where $\widehat{\text{Gal}(K/\mathbb{Q})}$ is the group of characters on $\text{Gal}(K/\mathbb{Q})$ and $\mathcal{L}_p(E,\alpha,\psi,T)$ is the $p$-adic $L$-function of $E/\mathbb{Q}$ twisted by $\psi$, as defined in Section \ref{notation}. From the interpolation property for each $\mathcal{L}_p(E,\alpha,\psi,T)$ and the assumption on the behaviour of the additive places in $K/\mathbb{Q}$, it is easy to see that $\mathcal{L}_p(E/K,\alpha,T)$ interpolates the values at $s=1$ of the complex $L$-function $L(E/K,s)$, twisted by characters of the $\mathbb{Z}_p$-cyclotomic extension of $K$. 

Furthermore, using Corollary \ref{functionalsupergeneralT} on each factor of the right hand side of (\ref{productformula}), we find the following transformation property for $\mathcal{L}_p(E/K,\alpha,T)$
\begin{equation}
\label{functionalgeneralT}
\mathcal{L}_p(E/K,\alpha,T)=\bigg[\prod_{\psi}(-c_{Q_{\psi}})\cdot \prod_{\psi}\bar{\psi}(-Q_{\psi})\bigg]\cdot \bigg[\prod_{\psi}\braket{Q_{\psi}}^{\frac{\log(1+T)^{-1}}{\log\kappa(\gamma)}}\bigg]\cdot \mathcal{L}_p(E/K,\alpha,(1+T)^{-1}-1),
\end{equation}
where the products run over all characters $\psi$ of $\widehat{\text{Gal}(K/\mathbb{Q})}$ and $Q_{\psi}$ is the largest divisor of $N$ coprime with $p$ and with the conductor of $\psi$.

Recall from Section \ref{notation} that the variables $T$ and $s$ are related by the identity $T=\kappa(\gamma)^{s-1}-1$. The analogue of Proposition \ref{functionalsupergenerals} for $E/K$ is then
\begin{equation}
\label{functionalgeneral}
L_p(E/K,\alpha, 2-s)=\bigg[\prod_{\psi}(-c_{Q_{\psi}})\cdot \prod_{\psi}\bar{\psi}(-Q_{\psi})\bigg]\cdot \bigg[\prod_{\psi}\braket{Q_{\psi}}^{s-1}\bigg]L_p(E/K,\alpha,s).
\end{equation}
Equations (\ref{functionalgeneral}) and (\ref{functionalgeneralT}) allow us to deduce analogues of Theorem \ref{mains} and Theorem \ref{main} for $E/K$. The proofs are essentially identical, so we include only the statements of the results.
\begin{thm}
\label{generalisations}
Let $E$ be an elliptic curve defined over $\mathbb{Q}$ of conductor $N$ and $p$ an odd prime of semistable reduction for $E$. Denote by $E/K$ the base-change of $E$ to an abelian number field $K$, where $K$ is assumed to be disjoint from the $\mathbb{Z}_p$-extension of $\mathbb{Q}$ and such that additive reduction is preserved everywhere in the extension $K/\mathbb{Q}$. For each allowable $p$-root $\alpha$, let $m$ be the order of vanishing of $L_p(E/K,\alpha,s)$ at $s=1$. Then
\begin{equation*}
L_p(E/K,\alpha,s)=a_m(s-1)^m+a_{m+1}(s-1)^{m+1}+\cdots
\end{equation*}
with
\begin{equation*}
a_{m+1}=-\frac{\sum_{\psi}\log\braket{Q_{\psi}}}{2}a_m,
\end{equation*}
where $\psi$ runs through the characters of $\text{Gal}(K/\mathbb{Q})$ and $Q_{\psi}$ is the largest divisor of $N$ which is coprime with $p$ and with the conductor of $\psi$.
\end{thm}
\begin{thm}
With the assumptions of Theorem \ref{generalisations}, we have
\begin{equation*}
\mathcal{L}_p(E/K,\alpha,T)=  c_m T^m+c_{m+1}T^{m+1}+\cdots
\end{equation*} 
with
\begin{equation*}
c_{m+1}=-\frac{c_m}{2}\left(\frac{\sum_{\psi}\log\braket{Q_{\psi}}}{\log\kappa(\gamma)}+m\right).
\end{equation*}
\end{thm}

\bibliographystyle{amsplain}
\bibliography{bibliomu}
\end{document}